\documentclass[11pt,final]{siamltex}
\usepackage{amsmath}
\usepackage{bm}
\usepackage{amssymb,version}
\usepackage{cases}
\usepackage{color}
\usepackage{verbatim}
\usepackage{multirow}
\usepackage{graphicx}

\allowdisplaybreaks
\begin{document}
    \title{Error estimate of a consistent splitting  GSAV scheme for the Navier-Stokes equations\thanks{The work of X. Li is supported by the National Natural Science Foundation of China  grants  12271302, 11971407. The work of J. Shen is supported in part by NSF grant DMS-2012585 and AFOSR  grant FA9550-20-1-0309. 
}}

    \author{ Xiaoli Li\thanks{School of Mathematics, Shandong University, Jinan, Shandong, 250100, P.R. China. Email: xiaolimath@sdu.edu.cn}.
        \and Jie Shen\thanks{Corresponding Author. Department of Mathematics, Purdue University, West Lafayette, IN 47907, USA. Email: shen7@purdue.edu}.
}

\maketitle

\begin{abstract}
	We carry out a rigorous error analysis of the first-order semi-discrete (in time) consistent splitting scheme  coupled with a generalized scalar auxiliary variable (GSAV) approach for the Navier-Stokes equations with no-slip boundary conditions. The scheme is linear,  unconditionally stable, and only requires solving a sequence of  Poisson type equations at each time step.  By using the build-in unconditional stability of the GSAV approach, we  derive  optimal global  (resp. local) in time  error estimates in the two (resp. three) dimensional case for the velocity and pressure approximations.  

\end{abstract}

\markboth{XIAOLI LI AND JIE SHEN } {Error estimate of a consistent splitting  GSAV scheme}
 \section{Introduction}
We consider  in this paper numerical approximation of the following time-dependent incompressible Navier-Stokes equations:
  \begin{subequations}\label{e_model}
    \begin{align}
     \frac{\partial \textbf{u}}{\partial t}+ ( \textbf{u}\cdot \nabla ) \textbf{u}
     -\nu\Delta\textbf{u}+\nabla p=\textbf{f}
     \quad &\ in\ \Omega\times J,
      \label{e_modelA}\\
      \nabla\cdot\textbf{u}=0
      \quad &\ in\ \Omega\times J,
      \label{e_modelB}\\
     \textbf{u}=\textbf{0} \quad &\ on\ \partial\Omega\times J,
      \label{e_modelC}
    \end{align}
  \end{subequations}
where $\Omega$ is an open bounded domain in $\mathbb{R}^d$ ($d=2,3$) with a sufficiently smooth boundary $\partial \Omega$, $J=(0,T]$, $(\textbf{u},p)$ represent the unknown  velocity and pressure,  $\textbf{f}$ is an external body force,
 $\nu>0$ is the  viscosity coefficient and $\textbf{n}$ is the unit outward normal  of the domain $\Omega$.
 
There exists a large  number of work devoted to the numerical approximations of the Navier-Stokes equations, see, for instance, \cite{girault1979finite,temam2001navier,glowinski2003finite,gunzburger2012finite}. and the references therein.  As we all known, the nonlinearity and the coupling of velocity and pressure have long been the main source of difficulties in both numerical analysis and simulations for the Navier-Stokes equations.  In view of numerical computation, it is desirable to treat the nonlinear term explicitly so that one only needs to solve simple linear equations with constant coefficients at each time step. However, a simple explicitly treatment usually leads to a severe stability constraint on the time step. 

The recently developed schemes \cite{lin2019numerical,li2020error,huang2021stability,li2022new,wu2022new}
 based on the scalar auxiliary variable (SAV) approach \cite{shen2019new}  with explicit treatment of nonlinear term can be unconditionally energy diminishing, and ample numerical results presented in the above work demonstrate that the above (implicit-explicit) IMEX type schemes (i.e. the nonlinear term is treated explicitly) are very  efficient and robust. 
 
On the other hand, there are in general two classes of numerical approaches to deal with the incompressible constraint: the coupled approach and the decoupled approach. The coupled approach is computationally expensive since it requires solving a saddle point problem at each time step \cite{girault1979finite,brezzi2012mixed,elman2014finite}. The decoupled approach, originated from the so called projection method \cite{chorin1968numerical,temam1969approximation}, can be  more effective for the reason that one only needs to solve a sequence of Poisson type equations to solve at each time step.  There have been extensive  efforts in   constructing various projection type schemes which 
 can be roughly classified  into  three categories \cite{GMS06}: the pressure-correction method \cite{shen1992error,weinan1995projection,guermond2004error,Shen2012Modeling,li2022new}, the velocity-correction method \cite{guermond2003velocity,serson2016velocity} and the consistent splitting method \cite{guermond2003new,Johnston2004Accurate,shen2007error} (see also the gauge method \cite{E2003Gauge,Nochetto2003Gauge}).  Among these, the consistent splitting scheme has outstanding advantages in the following two aspects: (i) It is free of the operator splitting error so it can achieve the full accuracy of the time discretization; (ii)  The inf-sup condition between the velocity and the pressure approximation spaces is not mandatory from a computational point of view.

 However, there are only very limited works on the stability and error analysis of the consistent splitting methods. Liu et al. derived a key inequality in  \cite{liu2007stability}  for the commutator between the Laplacian and Leray projection operators, and established local error estimates in \cite{liu2009error} for the first-order consistent splitting schemes in two- and three-dimensional cases. Recently, based on the generalized SAV approach \cite{MR4383075},
  Huang et al.  constructed high-order consistent splitting schemes for the Navier-Stokes equations with periodic boundary conditions in  \cite{huang2021stability} and no-slip boundary conditions in \cite{wu2022new}. Furthermore, in the case of  periodic boundary conditions, they  established  in  \cite{huang2021stability}   optimal error estimates (up to fith-order)  which are globally (resp. locally) in time  in the two (resp. three) dimensional case.  However,  it is non trivial to extend the corresponding error analysis to the case with non-periodic boundary conditions: (i) additional difficulty due to the pressure Poisson equation; and (ii) weaker  stability result compared  with the case of periodic boundary conditions. 
 
The main purpose of this paper is to carry out a rigorous error analysis for the first-order  consistent splitting SAV scheme for the Navier-Stokes equations with no-slip boundary conditions. 
	Our main contributions are:
\begin{itemize}
\item Global in time  error estimates in  $L^{\infty}(0,T;H^1(\Omega)) \bigcap L^{2}(0,T;H^2(\Omega))$ for the velocity and\newline
 $ L^{\infty}(0,T;L^2(\Omega)) \bigcap  L^{2}(0,T;H^1(\Omega))$ for the pressure  are established in the two-dimensional case. To the best of our knowledge, this appears to be the first global-in-time error estimate for a consisting splitting scheme with no-slip boundary conditions.
\item Local in time  error estimates in $L^{\infty}(0,T_*;H^1(\Omega)) \bigcap L^{2}(0,T_*;H^2(\Omega))$ $(T_* \leq T)$ for the velocity and $ L^{\infty}(0,T_*;L^2(\Omega)) \bigcap  L^{2}(0,T_*;H^1(\Omega))$ for the pressure  are established in the three-dimensional case. 
\end{itemize}

The paper is organized as follows. In Section 2,  we provide some preliminaries. In Section 3, we construct the first-order consistent-splitting scheme based on the SAV approach. In Section 4, we carry out a rigorous error estimates  for the first-order GSAV consistent splitting scheme. Some numerical experiments  are presented in Section 5 to validate our theoretical results.

  \section{Preliminaries}
We describe below some notations and results which will be frequently used in this paper.

Throughout the paper, we use $C$, with or without subscript, to denote a positive
constant, which could have different values at different appearances.

 Let $\Omega$ be an open bounded domain in $\mathbb{R}^d$ $(d=2,\,3)$, we will use the standard notations $L^2(\Omega)$, $H^k(\Omega)$ and $H^k_0(\Omega)$ to denote the usual Sobolev spaces over $\Omega$. The norm corresponding to $H^k(\Omega)$ will be denoted simply by $\|\cdot\|_k$. In particular, we use $\|\cdot\|$ to denote the norm in $L^2(\Omega)$. Besides, $(\cdot,\cdot)$ is used to denote the inner product in $L^2(\Omega)$, and boldface letters are used to denote  vector functions and vector spaces.

  We define
  $$\textbf{H}=\{ \textbf{v}\in \textbf{L}^2(\Omega): \text{div}\,\textbf{v}=0, \textbf{v}\cdot \textbf{n}|_{\Gamma}=0 \},\ \ \textbf{V}=\{\textbf{v}\in \textbf{H}^1_0(\Omega):  \text{div}\,\textbf{v}=0 \},$$
  and the Stokes operator
  $$ A\textbf{u}=-P_{H}\Delta\textbf{u},\ \ \forall \ \textbf{u}\in D(A)=\textbf{H}^2(\Omega)\cap\textbf{V},$$
where $P_{H}$ is the orthogonal projector in $\textbf{L}^2(\Omega)$ onto $\textbf{H}$ and the Stokes operator $A$ is an unbounded positive self-adjoint closed operator in $\textbf{H}$ with domain $D(A)$.

 We recall the inequalities below which will be used in the following \cite{temam2001navier,heywood1982finite}:
\begin{equation}\label{e_norm H2}
\aligned
\|\nabla\textbf{v}\|\leq c_1\|A^{\frac{1}{2}}\textbf{v}\|,\ \ \|\Delta\textbf{v}\|\leq c_1\|A\textbf{v}\|, \ \ \forall \textbf{v}\in D(A)=\textbf{H}^2(\Omega)\cap\textbf{V}.
\endaligned
\end{equation}
By using Poincar\'e inequality, we then derive from the above  that
\begin{equation}\label{e_norm H1}
\aligned
\|\textbf{v}\|\leq c_1\|\nabla\textbf{v}\|, \ \forall\textbf{v}\in \textbf{H}^1_0(\Omega),\ \
\|\nabla\textbf{v}\|\leq c_1\|A\textbf{v}\|, \ \ \forall \textbf{v}\in D(A),
\endaligned
\end{equation}
where $c_1$ is a positive constant depending only on $\Omega$.

Next the trilinear form $b(\cdot,\cdot,\cdot)$ is defined by
\begin{equation*}
\aligned
b(\textbf{u},\textbf{v},\textbf{w})=\int_{\Omega}(\textbf{u}\cdot\nabla)\textbf{v}\cdot \textbf{w}d\textbf{x}.
\endaligned
\end{equation*}
We can easily obtain that the trilinear form $b(\cdot,\cdot,\cdot)$ is a skew-symmetric with respect to its last two arguments, i.e.,
\begin{equation}\label{e_skew-symmetric1}
\aligned
b(\textbf{u},\textbf{v},\textbf{w})=-b(\textbf{u},\textbf{w},\textbf{v}),\ \ \forall  \ \textbf{u}\in \textbf{H}, \ \ \textbf{v}, \textbf{w}\in \textbf{H}^1_0(\Omega),
\endaligned
\end{equation}
and
\begin{equation}\label{e_skew-symmetric2}
\aligned
b(\textbf{u},\textbf{v},\textbf{v})=0,\ \ \forall \ \textbf{u}\in \textbf{H}, \ \ \textbf{v}\in \textbf{H}^1_0(\Omega).
\endaligned
\end{equation}

By applying a combination of integration by parts,  Holder's inequality,  and Sobolev inequalities \cite{Tema95,shen1992error}, we have that for $d\le 4$,

\begin{flalign}\label{e_estimate for trilinear form}
b(\textbf{u},\textbf{v},\textbf{w})\leq \left\{
   \begin{array}{l}
   c_2\|\textbf{u}\|_1\|\textbf{v}\|_1\|\textbf{w}\|_1,\ \ \forall  \ \textbf{u}, \textbf{v} \in \textbf{H}
   , \textbf{w}\in \textbf{H}^1_0(\Omega),\\
   c_2\|\textbf{u}\|_2\|\textbf{v}\|\|\textbf{w}\|_1, \ \ \forall \ \textbf{u}\in \textbf{H}^2(\Omega)\cap\textbf{H},\ \textbf{v} \in \textbf{H}, \textbf{w}\in \textbf{H}^1_0(\Omega),\\
   c_2\|\textbf{u}\|_2\|\textbf{v}\|_1\|\textbf{w}\|, \ \ \forall \ \textbf{u}\in \textbf{H}^2(\Omega)\cap\textbf{H},\ \textbf{v} \in \textbf{H}, \textbf{w}\in \textbf{H}^1_0(\Omega),\\
   c_2\|\textbf{u}\|_1\|\textbf{v}\|_2\|\textbf{w}\|, \ \ \forall \ \textbf{v}\in \textbf{H}^2(\Omega)\cap\textbf{H},\ \textbf{u}\in \textbf{H}, \textbf{w}\in \textbf{H}^1_0(\Omega),\\
   c_2\|\textbf{u}\|\|\textbf{v}\|_2\|\textbf{w}\|_1, \ \ \forall \ \textbf{v}\in \textbf{H}^2(\Omega)\cap\textbf{H},\ \textbf{u} \in \textbf{H}, \textbf{w} \in \textbf{H}^1_0(\Omega).
   \end{array}
   \right.
\end{flalign}
In addition, we have the following more precise inequalities \cite{temam1983nonlinear,huang2021stability}:
\begin{flalign}\label{e_estimate for trilinear form2}
b(\textbf{u},\textbf{v},\textbf{w})\leq \left\{
   \begin{array}{l}
    c_2\|\textbf{u}\|_1^{1/2}\|\textbf{u}\|^{1/2}\|\textbf{v}\|_1^{1/2}\|\textbf{v}\|_2^{1/2}\|\textbf{w}\|, \ \ \  d\le 2, \\
   c_2\|\textbf{u}\|_1 \| \nabla \textbf{v}\|_{1/2} \|\textbf{w}\|, \ \ \  d\le 3,
   \end{array}
   \right.
\end{flalign}
where $c_2$ is a positive constant depending only on $\Omega$.

We will frequently use the following discrete version of the Gronwall lemma \cite{shen1990long,HeSu07}:

\medskip
\begin{lemma} \label{lem: gronwall2}
Let $a_k$, $b_k$, $c_k$, $d_k$, $\gamma_k$, $\Delta t_k$ be nonnegative real numbers such that
\begin{equation}\label{e_Gronwall3}
\aligned
a_{k+1}-a_k+b_{k+1}\Delta t_{k+1}+c_{k+1}\Delta t_{k+1}-c_k\Delta t_k\leq a_kd_k\Delta t_k+\gamma_{k+1}\Delta t_{k+1}
\endaligned
\end{equation}
for all $0\leq k\leq m$. Then
 \begin{equation}\label{e_Gronwall4}
\aligned
a_{m+1}+\sum_{k=0}^{m+1}b_k\Delta t_k \leq \exp \left(\sum_{k=0}^md_k\Delta t_k \right)\{a_0+(b_0+c_0)\Delta t_0+\sum_{k=1}^{m+1}\gamma_k\Delta t_k \}.
\endaligned
\end{equation}
\end{lemma}
To obtain the local error estimates in the three-dimensional case, we recall the following lemma  \cite{liu2007stability,liu2010stable}:

\medskip
\begin{lemma} \label{lem: local gronwall}
Suppose that $F: (0, \infty) \rightarrow (0, \infty) $ is continuous and increasing, and let 
$T_*$ satisfy that  $0<T_*<  \int_{M}^{\infty} dx / F(x)$ with $M>0$. Suppose that quantities $x_n,  \ \omega_n \geq 0$ satisfy 
$$ x_n+ \sum_{k=0}^{n-1} \Delta t \omega(k) \leq M+  \sum_{k=0}^{n-1} \Delta t F(x_k), \ \forall n\leq n_*, $$
with $n_* \Delta t \leq T_*$. Then we have $M+  \sum_{k=0}^{n_*-1} \Delta t F(x_k) \leq C_*$, where $C_*$ is independent of $\Delta t$.

\end{lemma}

  \section{The first-order consistent-splitting scheme based on the SAV approach}
 In this section, we construct  the first-order consistent-splitting scheme based on the SAV approach for the Navier-Stokes equation \cite{wu2022new}.

Set $$\Delta t=T/N,\ t^n=n\Delta t, \ d_t g^{n+1}=\frac{g^{n+1}-g^n}{\Delta t},
\ {\rm for} \ n\leq N,$$
 and introduce a SAV
 \begin{equation}\label{e_definition of R}
\aligned
R(t)=E(\textbf{u})+K_0, \ E(\textbf{u}) = \frac1 2 \| \textbf{u} \|^2,
\endaligned
\end{equation}
with some  $K_0 >0$, and recast the governing system as the following equivalent form:
  \begin{numcases}{}
 \frac{\partial \textbf{u}}{\partial t}+ ( \textbf{u}\cdot \nabla )\textbf{u}
     -\nu\Delta\textbf{u}+\nabla p=\textbf{f},  \label{e_model_transform1}\\
  \frac{ d R}{ d t} =  \frac{R}{ E( \textbf{u} ) +K_0} \left(  -\nu \| \nabla \textbf{u} \|^2 + ( \textbf{f}, \textbf{u}) \right),   \label{e_model_transform2}\\
 \nabla\cdot\textbf{u}=0. \label{e_model_transform3}
\end{numcases}
It is clear that the above system is equivalent to the original system. Motivated by the SAV approach and the consistent splitting scheme, we construct the following first-order 
linear and decoupled scheme, which is almost the same to the first-order scheme in 
 \cite{wu2022new}:
Find ($\tilde{\textbf{u}}^{n+1}, \textbf{u}^{n+1}, p^{n+1}, \xi^{n+1}, R^{n+1} $) by solving
     \begin{eqnarray}
 &&   \frac{ \tilde{\textbf{u}}^{n+1} - \tilde{\textbf{u}}^{n}}{\Delta t}  -\nu\Delta\tilde{\textbf{u}}^{n+1} =  \textbf{f}^{n+1} - ( \textbf{u}^{n}\cdot \nabla ) \textbf{u}^{n} -\nabla p^{n}, \ \ \tilde{\textbf{u}}^{n+1}|_{\partial \Omega}=0; \label{e_model_semi1} \\
  &&   \frac{ R^{n+1} - R^n }{\Delta t} = \frac{R^{n+1} }{ E( \tilde{\textbf{u}}^{n+1} ) +K_0} \left(  -\nu \| \nabla \tilde{\textbf{u}}^{n+1} \|^2 + ( \textbf{f}^{n+1}, \tilde{\textbf{u}}^{n+1} ) \right);  \label{e_model_semi2}\\
 &&   \xi^{n+1} = \frac{R^{n+1} }{ E( \tilde{\textbf{u}}^{n+1} ) +K_0}, \ \eta^{n+1}=1-(1-\xi^{n+1})^2, \  \textbf{u}^{n+1}= \eta^{n+1} \tilde{\textbf{u}}^{n+1} ;
 \label{e_model_semi3} \\
 &&   (\nabla p^{n+1}, \nabla q) = \left(  \textbf{f}^{n+1} - ( \textbf{u}^{n+1} \cdot \nabla ) \textbf{u}^{n+1} - \nu \nabla \times \nabla \times \tilde{ \textbf{u} }^{n+1}, \nabla q) \right), \ \forall q\in H^1(\Omega). \label{e_model_semi4}
 \end{eqnarray}
 
 By using similar procedure in \cite{wu2022new}, we can easily obtain the following unconditional energy stability: 
 \medskip

 \begin{theorem}\label{thm_energy stability_first order}
 Let $\| \textbf{f}(\cdot,t) \| \leq M_{\textbf{f}}, \ \forall t \in [0,T]$, and $K_0\geq \max\{ 2M_{\textbf{f}}, 1 \}$. Then given $R^n>0$, we have  $\xi^{n+1}>0$ and 
 \begin{equation}\label{e_stability1}
\aligned
0<R^{n+1}<R^{n}, \ \ \forall n \leq T/\Delta t.
\endaligned
\end{equation} 
In addition, there exists a constant $C_T$ only depends on $T$ such that
\begin{equation}\label{e_stability2}
\aligned
\| \textbf{u}^n \| + \nu \sum_{k=0}^{n} \Delta t \xi^{k+1} \| \nabla \tilde{\textbf{u}}^{n+1} \|^2 
\leq C_T, \ \ \forall n \leq T/\Delta t,
\endaligned
\end{equation} 
 where $ \textbf{u}^n $ is the solution of scheme \eqref{e_model_semi1}-\eqref{e_model_semi4}.
\end{theorem}

  \section{Error Analysis}
In this section, we carry out a rigorous  error analysis for the first-order semi-discrete scheme \eqref{e_model_semi1}-\eqref{e_model_semi4} in two- and three-dimensional cases.

We  set
   \begin{numcases}{}
\displaystyle \tilde{e}_{\textbf{u}}^{n+1}=\tilde{\textbf{u}}^{n+1}-\textbf{u}(t^{n+1}),\ \
\displaystyle e_{\textbf{u}}^{n+1}=\textbf{u}^{n+1}-\textbf{u}(t^{n+1}), \notag\\
\displaystyle e_{p}^{n+1}=p^{n+1}-p(t^{n+1}),\ \ \
\displaystyle e_{R}^{n+1}=R^{n+1}-R(t^{n+1}).\notag
\end{numcases}
Next we give some preliminaries to estimate the part of pressure. Similar to \cite{liu2007stability},  we let $\mathcal{P}$ denote the Leray-Helmholtz projection operator onto divergence-free fields, defined as follows. Given any $\textbf{b} \in L^2(\Omega, \mathbb{R}^d)$, there is a unique $q \in H^1(\Omega) $ with $ \int_{\Omega} q =0 $ such that $ \mathcal{P} \textbf{b} =  \textbf{b} + \nabla q$ satisfies 
\begin{equation}\label{e_error_stokespre1}
\aligned
(\textbf{b} + \nabla q, \nabla \phi ) = ( \mathcal{P} \textbf{b}, \nabla \phi )=0, \ \forall \phi \in 
H^1(\Omega).
\endaligned
\end{equation} 
Then for $\textbf{u} \in L^2(\Omega, \mathbb{R}^d)$, we have  \cite{liu2007stability}
\begin{equation}\label{e_error_stokespre2}
\aligned
\Delta \mathcal{P} \textbf{u} = \Delta \textbf{u} - \nabla \nabla \cdot \textbf{u} = - \nabla \times \nabla \times  \textbf{u}.
\endaligned
\end{equation} 
Next we recall the estimate for commutator of the Laplacian and Leray-Helmholtz projection operators. 
\begin{lemma}\label{lem: commutator}
\cite{liu2007stability} Let $\Omega \subset \mathbb{R}^d $ be a connected bounded domain with $C^3$ boundary. Then for any $\epsilon >0$, there exists a positive constant $C \geq 0$ such that for all vector fields $\textbf{u} \in H^2 \cap H^1_0 (\Omega, \mathbb{R}^d )$,
\begin{equation}\label{e_error_stokespre3}
\aligned
\int_{\Omega} | (\Delta \mathcal{P} - \mathcal{P} \Delta) \textbf{u} |^2 \leq (\frac1 2 + \epsilon ) \int_{\Omega} |\Delta \textbf{u} |^2 + C \int_{\Omega} |\nabla \textbf{u} |^2.
\endaligned
\end{equation} 
\end{lemma}

We define the Stokes pressure $p_s(\textbf{u})$ by 
\begin{equation}\label{e_error_stokespre4}
\aligned
\nabla p_s(\textbf{u}) =  (\Delta \mathcal{P} - \mathcal{P} \Delta) \textbf{u},
\endaligned
\end{equation} 
where the Stokes pressure is generated by the tangential part of vorticity at the boundary in two and three dimensions by \cite{liu2007stability,liu2009error}
\begin{equation}\label{e_error_stokespre5}
\aligned
\int_{\Omega} \nabla p_s(\textbf{u}) \cdot \nabla \phi = \int_{\Gamma} (\nabla \times \textbf{u} ) \cdot ( \textbf{n} \times \nabla \phi ), \ \forall \phi \in H^1(\Omega).
\endaligned
\end{equation} 
Then by using \eqref{e_error_stokespre2}, we have 
\begin{equation}\label{e_error_stokespre6}
\aligned
\nabla p_s(\textbf{u}) =  (\Delta \mathcal{P} - \mathcal{P} \Delta) \textbf{u} = (I - \mathcal{P} ) \Delta  \textbf{u} - \nabla \nabla \cdot \textbf{u} = (I - \mathcal{P} ) ( \Delta  \textbf{u} - \nabla \nabla \cdot \textbf{u} ). 
\endaligned
\end{equation} 
Recalling \eqref{e_error_stokespre1}, we have 
\begin{equation}\label{e_error_stokespre7}
\aligned
\int_{\Omega} \nabla p_s(\textbf{u}) \cdot \nabla \phi = \int_{\Omega} ( \Delta \textbf{u} - \nabla \nabla \cdot \textbf{u} ) \cdot \nabla \phi, \ \forall \phi \in H^1(\Omega).
\endaligned
\end{equation} 

The main result of this section is stated in the following  theorem.
\medskip
\begin{theorem}\label{the: error_estimate_final}
Assume $\textbf{u}\in H^2(0,T;\textbf{L}^2(\Omega))\bigcap H^1(0,T;\textbf{H}^2(\Omega))\bigcap L^{\infty}(0,T; \textbf{H}^2(\Omega))$, $p\in H^1(0,T;L^2(\Omega))$ and 
$$  \Delta t \sum\limits_{k=0}^{n+1} \|  \textbf{f}^{k} \|^2 +  \Delta t \| \Delta \textbf{u}^{0} \|^2 +  \| \nabla \textbf{u}^{0} \|^2 \leq C^* $$
with $C^*>0$,
 then for the first-order  scheme \eqref{e_model_semi1}-\eqref{e_model_semi4} with $\Delta t \leq \frac{1}{ 1+C_0^2 }$, we have
\begin{equation}\label{e_error_estimate_u_final}
\aligned
& \|  \textbf{u}^{n+1} - \textbf{u}(t^{n+1})  \|^2 + \|  \nabla (  \textbf{u}^{n+1} - \textbf{u}(t^{n+1}) )  \|^2  + \Delta t \sum\limits_{k=0}^{n} \| \Delta ( \textbf{u}^{k+1} - \textbf{u}(t^{k+1})) \|^2  \\
\leq &
\begin{cases}
C (\Delta t)^2, \ \ d=2, \ \forall n\leq T /\Delta t,\\
C (\Delta t)^2,  \ \ d=3, \ \forall n\leq T_*/\Delta t,
\end{cases}
\endaligned
\end{equation}
and
\begin{equation}\label{e_error_estimate_p_final}
\aligned 
 \| \nabla (p^{n+1}- p(t^{n+1}) ) \|^2 +  \Delta t \sum\limits_{k=0}^{n} \| \nabla (p^{k+1}- p(t^{k+1}) ) \|^2 
\leq &
\begin{cases}
C (\Delta t)^2, \ \ d=2, \ \forall n\leq T/\Delta t, \\
C (\Delta t)^2,  \ \ d=3, \ \forall n\leq T_*/\Delta t,
\end{cases}
\endaligned
\end{equation} 
where $T_*$ is defined in \eqref{e_error_velo14} and the constants
$C_0$ and $C$  are some positive constants independent of $\Delta t$.
\end{theorem}

\begin{proof}
First we shall make the hypothesis that there exists a positive constant $C_0$ such that
\begin{equation}\label{e_error_velo1}
\aligned
| 1- \xi^k | \leq C_0 \Delta t, \ \forall k \leq T/ \Delta t,
\endaligned
\end{equation}
which will be proved in the induction process below by using a bootstrap argument.

We can easily obtain that \eqref{e_error_velo1} holds for $k=0$. Now we suppose 
\begin{equation}\label{e_error_velo2}
\aligned
| 1- \xi^k | \leq C_0 \Delta t, \ \forall k \leq n,
\endaligned
\end{equation}
and we shall prove that  $| 1- \xi^{n+1} | \leq C_0 \Delta t$ holds true.

\noindent{\bf Step 1:  $H^2$ bounds for $\tilde{\textbf{u}} ^k$ and $\textbf{u} ^k$ with $k \leq n$ in two- and three-dimensional cases.} 
First using exactly the same procedure in \cite{huang2021stability}, we can easily obtain that 
\begin{equation}\label{e_error_velo3}
\aligned
\frac{1}{2} \leq | \xi^k |, \  | \eta^k |  \leq 2, 
\endaligned
\end{equation} 
 under the condition $\Delta t \leq \min\{ \frac{1}{4C_0},1\}$. Recalling Theorem \ref{thm_energy stability_first order}, we have 
\begin{equation}\label{e_error_velo4}
\aligned
\| \tilde{\textbf{u}}^k \| + \nu \sum_{l=0}^{n} \Delta t  \| \nabla \tilde{\textbf{u}}^{k+1} \|^2
 + \nu \sum_{l=0}^{n} \Delta t  \| \nabla \textbf{u}^{k+1} \|^2 
 \leq 8 C_T, \ \ C_0 \geq 1,
\endaligned
\end{equation} 
where $C_T$ is independent of $C_0$.

Noting 
\begin{equation}\label{e_error_velo5}
\aligned
& \Delta \textbf{u} - \nabla \nabla \cdot \textbf{u} = - \nabla \times \nabla \times \textbf{u},
\endaligned
\end{equation} 
and taking $ q = p^{k+1} $ in \eqref{e_model_semi4} lead to
\begin{equation}\label{e_error_velo6_pre}
\aligned
 \| \nabla p^{k+1} \| \leq & \|  \textbf{f}^{k+1} - ( \textbf{u}^{k+1} \cdot \nabla ) \textbf{u}^{k+1} \| + \nu \|  \nabla p_s^{k+1}( \tilde{\textbf{u}} ) \|.
\endaligned
\end{equation} 
Recalling \eqref{e_error_stokespre7} and lemma \ref{lem: commutator}, we have
\begin{equation}\label{e_error_velo6}
\aligned
 \nu \|  \nabla p_s^{k+1}( \tilde{\textbf{u}} ) \|^2 
\leq & \nu \alpha \| \Delta \tilde{\textbf{u}}^{k+1} \|^2 + \nu  C_{\alpha} \| \nabla  \tilde{\textbf{u}}^{k+1} \|^2,
\endaligned
\end{equation} 
where the positive constant $ \frac1 2 < \alpha < 1$.

Taking the inner product of \eqref{e_model_semi1} with $-2 \Delta t \Delta \tilde{\textbf{u}}^{k+1}$, we obtain
\begin{equation}\label{e_error_velo7}
\aligned
&( \| \nabla \tilde{\textbf{u}}^{k+1} \|^2 - \| \nabla \tilde{\textbf{u}}^{k} \|^2 + \| \nabla \tilde{\textbf{u}}^{k+1} - \nabla \tilde{\textbf{u}}^{k} \|^2) + 2 \nu \Delta t \| \Delta \tilde{\textbf{u}}^{k+1} \|^2 \\
\leq & 2\Delta t \| \Delta \tilde{\textbf{u}}^{k+1} \| \| \textbf{f}^{k+1} - ( \textbf{u}^{k}\cdot \nabla ) \textbf{u}^{k} \| + 2\Delta t \| \Delta \tilde{\textbf{u}}^{k+1} \| \|  \nabla p^{k} \| \\
\leq & 2\Delta t \| \Delta \tilde{\textbf{u}}^{k+1} \| ( \| \textbf{f}^{k+1} - ( \textbf{u}^{k}\cdot \nabla ) \textbf{u}^{k} \|  + \| \textbf{f}^{k} - ( \textbf{u}^{k}\cdot \nabla ) \textbf{u}^{k} \| + \nu  \|  \nabla p_s^{k}( \tilde{\textbf{u}} ) \| ) \\
\leq &\nu  \Delta t \| \Delta \tilde{\textbf{u}}^{k+1} \|^2 +  \nu \alpha \Delta t  \|  \Delta \tilde{\textbf{u}}^{k} \|^2  + \nu \Delta t C_{\alpha} \| \nabla  \tilde{\textbf{u}}^{k} \|^2 + \frac{1-\alpha}{4} \nu  \Delta t \| \Delta \tilde{\textbf{u}}^{k+1} \|^2\\ 
& + \frac{8}{(1-\alpha)\nu} \Delta t ( \|  \textbf{f}^{k} \|^2 + \|  \textbf{f}^{k+1} \|^2 ) +  \frac{16}{(1-\alpha)\nu} \Delta t \| ( \textbf{u}^{k} \cdot \nabla ) \textbf{u}^{k} \|^2.
\endaligned
\end{equation} 
Next we shall estimate the nonlinear term. Taking notice of the fact that by  using  Ladyzhenskaya's inequalities and Sobolev embedding theorems \cite{ladyzhenskaya1969mathematical,liu2007stability} and \eqref{e_estimate for trilinear form2}, we have 
\begin{equation}\label{e_error_velo8}
\aligned
\| ( \textbf{u}^{k} \cdot \nabla ) \textbf{u}^{k} \|^2 \leq 
\begin{cases}
\| \textbf{u}^{k}  \|^2_{L^4} \| \nabla \textbf{u}^{k}  \|^2_{L^4} \leq C \| \textbf{u}^{k}  \|_{L^2} \| \nabla \textbf{u}^{k}  \|_{L^2}^2 \| \nabla \textbf{u}^{k}  \|_{H^1}, \ \ d=2, \\
\| \textbf{u}^{k}  \|^2_{L^6} \| \nabla \textbf{u}^{k}  \|^2_{L^3} \leq C \| \nabla \textbf{u}^{k}  \|_{L^2}^3 \| \nabla \textbf{u}^{k}  \|_{H^1}, \ \ d=3.
\end{cases}
\endaligned
\end{equation} 
In addition, by using the elliptic regularity estimate $ \| \textbf{u}^{k} \|_{H^2} \leq C \| \Delta \textbf{u}^{k} \| $ and recalling $ \textbf{u}^{k}= \eta^{k} \tilde{\textbf{u}}^{k}$, we have
\begin{equation}\label{e_error_velo9}
\aligned
\| ( \textbf{u}^{k} \cdot \nabla ) \textbf{u}^{k} \|^2 \leq 
\begin{cases}
\frac{(1-\alpha)^2 \nu^2}{64} \| \Delta \tilde{\textbf{u}}^{k} \|^2 + C \| \textbf{u}^{k} \|^2 \| \nabla \tilde{\textbf{u}}^{k} \|^4, \ \ d=2, \\
\frac{(1-\alpha)^2 \nu^2}{64} \| \Delta \tilde{\textbf{u}}^{k} \|^2 + C \| \nabla \tilde{\textbf{u}}^{k} \|^6, \ \ d=3.
\end{cases}
\endaligned
\end{equation} 

Thus for $d=2$, we can recast \eqref{e_error_velo7} as follows:
\begin{equation}\label{e_error_velo10}
\aligned
&( \| \nabla \tilde{\textbf{u}}^{k+1} \|^2 - \| \nabla \tilde{\textbf{u}}^{k} \|^2 + \| \nabla \tilde{\textbf{u}}^{k+1} - \nabla \tilde{\textbf{u}}^{k} \|^2) +  \nu \Delta t \| \Delta \tilde{\textbf{u}}^{k+1} \|^2 \\
\leq & \frac{1-\alpha}{4} \nu  \Delta t \| \Delta \tilde{\textbf{u}}^{k+1} \|^2 +\frac{1-\alpha}{4} \nu  \Delta t \| \Delta \tilde{\textbf{u}}^{k} \|^2 +  \nu \alpha \Delta t  \|  \Delta \tilde{\textbf{u}}^{k} \|^2  + \nu \Delta t C_{\alpha} \| \nabla  \tilde{\textbf{u}}^{k} \|^2 
\\ 
& + C \Delta t ( \|  \textbf{f}^{k} \|^2 + \|  \textbf{f}^{k+1} \|^2 ) + C \Delta t \| \textbf{u}^{k} \|^2 \| \nabla \tilde{\textbf{u}}^{k} \|^2 \| \nabla \tilde{\textbf{u}}^{k} \|^2 \\
\leq &  \frac{1-\alpha}{4}  \nu \Delta t \| \Delta \tilde{\textbf{u}}^{k+1} \|^2   + ( \nu  C_{\alpha} +  C \| \textbf{u}^{k} \|^2 \| \nabla \tilde{\textbf{u}}^{k} \|^2 ) \Delta t \| \nabla  \tilde{\textbf{u}}^{k} \|^2  \\ 
& + (\frac{1-\alpha}{4} \nu + \alpha \nu) \Delta t \| \Delta \tilde{\textbf{u}}^{k} \|^2 + C \Delta t ( \|  \textbf{f}^{k} \|^2 + \|  \textbf{f}^{k+1} \|^2 ) .
\endaligned
\end{equation} 
Summing \eqref{e_error_velo10} over $k$, $k=0,2,\ldots,n$, using \eqref{e_error_velo4}
and lemma \ref{lem: gronwall2}, we can arrive at
\begin{equation}\label{e_error_velo11}
\aligned
& \| \nabla \tilde{\textbf{u}}^{n+1} \|^2 + \Delta t \sum\limits_{k=0}^{n} \| \Delta \tilde{\textbf{u}}^{k+1} \|^2  \\
& \ \ \ \ \ \ 
\leq  C \Delta t \sum\limits_{k=0}^{n+1} \|  \textbf{f}^{k} \|^2 + C \Delta t \| \Delta \tilde{\textbf{u}}^{0} \|^2 +C  \| \nabla \tilde{\textbf{u}}^{0} \|^2 \leq C^*, \ \ d=2,
\endaligned
\end{equation}  
where $C^*$ is independent of $\Delta t$ and $ C_0$. Recalling \eqref{e_error_velo3}, we have for $d=2$,
\begin{equation}\label{e_error_velo15}
\aligned
&  \| \nabla \textbf{u}^{n} \|^2 + \Delta t \sum\limits_{k=0}^{n} \| \Delta \textbf{u}^{k} \|^2
\leq C, 
\endaligned
\end{equation}
where $C$ is independent of $\Delta t$ and $C_0$. 

Next we consider the case with $d=3$. Using \eqref{e_error_velo9}, we can transform \eqref{e_error_velo7} into the following:
\begin{equation}\label{e_error_velo12}
\aligned
&( \| \nabla \tilde{\textbf{u}}^{k+1} \|^2 - \| \nabla \tilde{\textbf{u}}^{k} \|^2 + \| \nabla \tilde{\textbf{u}}^{k+1} - \nabla \tilde{\textbf{u}}^{k} \|^2) +  \nu \Delta t \| \Delta \tilde{\textbf{u}}^{k+1} \|^2 \\
\leq & \frac{1-\alpha}{4} \nu  \Delta t \| \Delta \tilde{\textbf{u}}^{k+1} \|^2 +\frac{1-\alpha}{4} \nu  \Delta t \| \Delta \tilde{\textbf{u}}^{k} \|^2 +  \nu \alpha \Delta t  \|  \Delta \tilde{\textbf{u}}^{k} \|^2  + \nu \Delta t C_{\alpha} \| \nabla  \tilde{\textbf{u}}^{k} \|^2 
\\ 
& + C \Delta t ( \|  \textbf{f}^{k} \|^2 + \|  \textbf{f}^{k+1} \|^2 ) +  C \Delta t  \| \nabla \tilde{\textbf{u}}^{k} \|^6.
\endaligned
\end{equation} 
Summing \eqref{e_error_velo12} over $k$, $k=0,2,\ldots,n$ implies that
\begin{equation}\label{e_error_velo13}
\aligned
& \| \nabla \tilde{\textbf{u}}^{n+1} \|^2 +  \Delta t \sum\limits_{k=0}^{n} \| \Delta \tilde{\textbf{u}}^{k+1} \|^2  \\
& \ \ \ \ \ \ 
\leq  C \Delta t \sum\limits_{k=0}^{n} \| \nabla \tilde{\textbf{u}}^{k} \|^6 + C \Delta t \sum\limits_{k=0}^{n+1} \|  \textbf{f}^{k} \|^2 + C \Delta t \| \Delta \tilde{\textbf{u}}^{0} \|^2 +C  \| \nabla \tilde{\textbf{u}}^{0} \|^2, \ \ d=3,
\endaligned
\end{equation}  
Recalling   lemma \ref{lem: local gronwall}, we let $M_0>0$ and  $F(x)=x^6$,  and choose $T^*$ satisfy that  $0<T^*<  \int_{M_0}^{\infty} dx / F(x)$, then we can estimate 
\eqref{e_error_velo13} as follows:
\begin{equation}\label{e_error_velo14}
\aligned
& \| \nabla \tilde{\textbf{u}}^{n+1} \|^2 +  \| \nabla \textbf{u}^{n+1} \|^2 + \Delta t \sum\limits_{k=0}^{n} \| \Delta \tilde{\textbf{u}}^{k+1} \|^2   + \Delta t \sum\limits_{k=0}^{n} \| \Delta \textbf{u}^{k+1} \|^2
\leq  C_*, \ \ d=3, \ \forall n\leq T_*/\Delta t,
\endaligned
\end{equation}
where $T_* = \min\{ T^*, T\}$ and $C_*$ is independent of $\Delta t$ and $C_0$.

\noindent{\bf Step 2: Estimates for $H^2$ bounds for $\tilde{e}_{\textbf{u}}^{n+1}$ in two- and three-dimensional cases.} 

We shall first start by establishing  an error equation corresponding to \eqref{e_model_semi1}. 
 Let $\textbf{S}_{\textbf{u}}^{k+1}$ be the truncation error defined by
\begin{equation}\label{e_error_step2_1}
\aligned
\textbf{S}_{\textbf{u}}^{k+1}=\frac{\partial \textbf{u}(t^{k+1})}{\partial t}- \frac{\textbf{u}(t^{k+1})-\textbf{u}(t^{k})}{\Delta t}=\frac{1}{\Delta t}\int_{t^k }^{t^{k+1}}(t^k-t)\frac{\partial^2 \textbf{u}}{\partial t^2}dt.
\endaligned
\end{equation}
Subtracting \eqref{e_model_transform1} at $t^{k+1}$ from \eqref{e_model_semi1}, we obtain
\begin{equation}\label{e_error_step2_2}
\aligned
&\frac{\tilde{e}_{\textbf{u}}^{k+1}- \tilde{e}_{\textbf{u}}^k}{\Delta t}-\nu\Delta\tilde{e}_{\textbf{u}}^{k+1}
=(\textbf{u}(t^{k+1})\cdot \nabla)\textbf{u}(t^{k+1}) \\
&\ \ \ \ \ \ \ \ \
- (\textbf{u}^{k}\cdot \nabla )\textbf{u}^{k}
-\nabla (p^k-p(t^{k+1}))+\textbf{S}_{\textbf{u}}^{k+1}.
\endaligned
\end{equation}
Next we establish an error equation for pressure corresponding to \eqref{e_model_semi4} by
\begin{equation}\label{e_error_step2_3}
\aligned
(\nabla e_p^{k+1}, \nabla q) = & \left(  ( \textbf{u}(t^{k+1} ) \cdot \nabla ) \textbf{u}( t^{k+1} ) - ( \textbf{u}^{k+1} \cdot \nabla ) \textbf{u}^{k+1} , \nabla q) \right) \\
& - ( \nu \nabla \times \nabla \times \tilde{e}_\textbf{u} ^{k+1}, \nabla q), \ \forall q\in H^1(\Omega).
\endaligned
\end{equation}
Taking $ q = e_p^{k+1} $ in \eqref{e_error_step2_3} leads to
\begin{equation}\label{e_error_step2_pre}
\aligned
 \| \nabla e_p^{k+1} \| \leq & \|   ( \textbf{u}(t^{k+1} ) \cdot \nabla ) \textbf{u}( t^{k+1} )  - ( \textbf{u}^{k+1} \cdot \nabla ) \textbf{u}^{k+1} \| + \nu \|  \nabla e_{ps}^{k+1}( \tilde{e}_\textbf{u} ) \| \\
\leq &  \nu \|  \nabla e_{ps}^{k+1}( \tilde{e}_\textbf{u} ) \|+  \|  ( e_{ \textbf{u} }^{k+1} \cdot \nabla ) \textbf{u}(t^{k+1} )  \| \\
& + \| ( \textbf{u}^{k+1} \cdot \nabla ) e_{ \textbf{u} }^{k+1} \| 
\endaligned
\end{equation} 
Recalling \eqref{e_error_stokespre7} and lemma \ref{lem: commutator}, we have
\begin{equation}\label{e_error_step2_4}
\aligned
\nu \|  \nabla e_{ps}^{k+1}( \tilde{e}_\textbf{u} ) \|^2 
\leq & \nu \alpha \| \Delta \tilde{e}_\textbf{u} ^{k+1} \|^2 + \nu  C_{\alpha} \| \nabla  \tilde{e}_\textbf{u} ^{k+1} \|^2 ,
\endaligned
\end{equation} 
where the positive constant $ \frac1 2 < \alpha < 1$.
Taking the inner product of \eqref{e_error_step2_2} with $-2 \Delta t \Delta \tilde{e}_{\textbf{u}}^{k+1}$, we obtain
\begin{equation}\label{e_error_step2_5}
\aligned
& ( \| \nabla \tilde{e}_{\textbf{u}}^{k+1} \|^2 - \| \nabla \tilde{e}_{\textbf{u}}^{k} \|^2 + \| \nabla \tilde{e}_{\textbf{u}}^{k+1}- \nabla \tilde{e}_{\textbf{u}}^{k} \|^2) + 2 \nu \Delta t \| \Delta \tilde{e}_{\textbf{u}}^{k+1} \|^2 \\
\leq & 2\Delta t \| \Delta \tilde{e}_{\textbf{u}}^{k+1}  \|  \| (\textbf{u}(t^{k+1})\cdot \nabla)\textbf{u}(t^{k+1}) - ( \textbf{u}^{k}\cdot \nabla ) \textbf{u}^{k} \| \\
& + 2\Delta t \| \Delta \tilde{e}_{\textbf{u}}^{k+1}  \| \| \nabla e_p^k \| +  2\Delta t \| \Delta \tilde{e}_{\textbf{u}}^{k+1}  \| ( \|p(t^n)-p(t^{k+1}) \|+\| \textbf{S}_{\textbf{u}}^{k+1} \| ) \\
\leq &  2\Delta t \| \Delta \tilde{e}_{\textbf{u}}^{k+1}  \| \| \nabla e_p^k \|  + 2\Delta t \| \Delta \tilde{e}_{\textbf{u}}^{k+1}  \|  (  \|  ( e_{ \textbf{u} }^{k} \cdot \nabla ) \textbf{u}(t^{k} )  \| + \| ( \textbf{u}^{k} \cdot \nabla ) e_{ \textbf{u} }^{k} \| ) \\
& +  2\Delta t \| \Delta \tilde{e}_{\textbf{u}}^{k+1}  \| ( \|p(t^n)-p(t^{k+1}) \|+\| \textbf{S}_{\textbf{u}}^{k+1} \|+ \| (\textbf{u}(t^{k+1})\cdot \nabla)\textbf{u}(t^{k+1}) - ( \textbf{u}(t^{k})\cdot \nabla)\textbf{u}(t^{k})  \|  ) \\
\leq & 2 \nu \Delta t \| \Delta \tilde{e}_{\textbf{u}}^{k+1}  \| \|  \nabla e_{ps}^{k}( \tilde{e}_\textbf{u} ) \| + 3 \Delta t \| \Delta \tilde{e}_{\textbf{u}}^{k+1}  \|  (  \|  ( e_{ \textbf{u} }^{k} \cdot \nabla ) \textbf{u}(t^{k} )  \| + \| ( \textbf{u}^{k} \cdot \nabla ) e_{ \textbf{u} }^{k} \| ) \\
& +  2\Delta t \| \Delta \tilde{e}_{\textbf{u}}^{k+1}  \| ( \|p(t^n)-p(t^{k+1}) \|+\| \textbf{S}_{\textbf{u}}^{k+1} \|+ \| (\textbf{u}(t^{k+1})\cdot \nabla)\textbf{u}(t^{k+1}) - ( \textbf{u}(t^{k})\cdot \nabla)\textbf{u}(t^{k})  \|  ).
\endaligned
\end{equation}
Using Cauchy-Schwarz inequality, the first term on the right hand side of \eqref{e_error_step2_5} can be estimated by
\begin{equation}\label{e_error_step2_6}
\aligned
& 2 \nu \Delta t \| \Delta \tilde{e}_{\textbf{u}}^{k+1}  \| \|  \nabla e_{ps}^{k}( \tilde{e}_\textbf{u} ) \| \leq \nu \Delta t  \| \Delta \tilde{e}_{\textbf{u}}^{k+1}  \|^2 +  \nu \alpha \Delta t \| \Delta \tilde{e}_\textbf{u} ^{k} \|^2+ \nu  C_{\alpha} \Delta t \| \nabla  \tilde{e}_\textbf{u} ^{k} \|^2.
\endaligned
\end{equation}
Recalling the Sobolev embedding theorems and Ladyzhenskaya’s inequalities, we have 
\begin{equation}\label{e_error_nonlinear}
\aligned
\int_{\Omega} | (\textbf{f} \cdot \nabla ) \textbf{g} |^2 \leq \left( \int_{\Omega} | \textbf{f} |^6 \right)^{1/3}  \left( \int_{\Omega} | \nabla \textbf{g} |^3 \right)^{2/3} \leq C \| \nabla \textbf{f} \|^2 \| \nabla \textbf{g} \| \| \nabla \textbf{g} \|_{H^1} .  
\endaligned
\end{equation}
Thanks to \eqref{e_error_nonlinear} and the $H^2$ boundedness for $ \tilde{\textbf{u}}^{k} $ in \eqref{e_error_velo11} and \eqref{e_error_velo14}, the second term on the right hand side of \eqref{e_error_step2_5} can be estimated by
\begin{equation}\label{e_error_step2_7}
\aligned
& 3\Delta t \| \Delta \tilde{e}_{\textbf{u}}^{k+1}  \|  (  \|  ( e_{ \textbf{u} }^{k} \cdot \nabla ) \textbf{u}(t^{k} )  \| + \| ( \textbf{u}^{k} \cdot \nabla ) e_{ \textbf{u} }^{k} \| ) \\
\leq & \frac{ (1-\alpha) \nu}{6} \Delta t  \| \Delta \tilde{e}_{\textbf{u}}^{k+1} \|^2 + C\Delta t   \|  ( e_{ \textbf{u} }^{k} \cdot \nabla ) \textbf{u}(t^{k} )  \|^2 + C \Delta t  \| ( \textbf{u}^{k} \cdot \nabla ) e_{ \textbf{u} }^{k} \|^2 \\
\leq &  \frac{ (1-\alpha) \nu}{6} \Delta t  \| \Delta \tilde{e}_{\textbf{u}}^{k+1} \|^2 + C \Delta t  \|  ( e_{ \textbf{u} }^{k} \cdot \nabla ) \textbf{u}(t^{k} )  \|^2 + C \Delta t  \| ( \textbf{u}^{k} \cdot \nabla ) e_{ \textbf{u} }^{k} \|^2 \\
\leq &  \frac{ (1-\alpha) \nu}{6}  \Delta t  \| \Delta \tilde{e}_{\textbf{u}}^{k+1} \|^2 + C \Delta t  \| \nabla e_{ \textbf{u} }^{k} \|^2 \| \nabla \textbf{u}(t^{k} ) \|  \| \nabla \textbf{u}(t^{k})  \|_{H^1} + C \Delta t \| \nabla \textbf{u}^{k} \|^2 \| \nabla e_{ \textbf{u} }^{k}  \|  \| \nabla e_{ \textbf{u} }^{k}  \|_{H^1}  \\
\leq & \frac{ (1-\alpha) \nu}{6} \Delta t  \| \Delta \tilde{e}_{\textbf{u}}^{k+1} \|^2 +   \frac{ (1-\alpha) \nu}{12} \Delta t  \| \Delta e_{\textbf{u}}^{k} \|^2 +  C \Delta t  \| \nabla e_{ \textbf{u} }^{k} \|^2 .
\endaligned
\end{equation}
From \eqref{e_model_semi3} and \eqref{e_error_velo2}, we have
\begin{equation}\label{e_error_step2_8}
\aligned
  \| \Delta e_{\textbf{u}}^{k} \|^2 \leq&  2 \| \Delta \tilde{e}_{\textbf{u}}^{k} \|^2 + 2 |1- \eta^{k} |^2 \| \Delta \tilde{\textbf{u}}^{k} \|^2 \\
\leq & 2 \| \Delta \tilde{e}_{\textbf{u}}^{k} \|^2 + 2  \| \Delta \tilde{\textbf{u}}^{k} \|^2 C_0^4 (\Delta t)^4,
\endaligned
\end{equation}
and 
\begin{equation}\label{e_error_step2_9}
\aligned
  \| \nabla e_{\textbf{u}}^{k} \|^2 \leq&  2 \| \nabla \tilde{e}_{\textbf{u}}^{k} \|^2 + 2 |1- \eta^{k} |^2 \| \nabla \tilde{\textbf{u}}^{k} \|^2 \\
\leq & 2 \| \nabla \tilde{e}_{\textbf{u}}^{k} \|^2 + 2  \| \nabla \tilde{\textbf{u}}^{k} \|^2 C_0^4 (\Delta t)^4.
\endaligned
\end{equation}
Thus we can recast \eqref{e_error_step2_7} as
\begin{equation}\label{e_error_step2_10}
\aligned
& 3\Delta t \| \Delta \tilde{e}_{\textbf{u}}^{k+1}  \|  (  \|  ( e_{ \textbf{u} }^{k} \cdot \nabla ) \textbf{u}(t^{k} )  \| + \| ( \textbf{u}^{k} \cdot \nabla ) e_{ \textbf{u} }^{k} \| ) \\
\leq & \frac{ (1-\alpha) \nu}{6} \Delta t  \| \Delta \tilde{e}_{\textbf{u}}^{k+1} \|^2 +   \frac{ (1-\alpha) \nu}{6}  \Delta t  \| \Delta \tilde{e}_{\textbf{u}}^{k} \|^2 + C \Delta t  \| \nabla \tilde{e}_{\textbf{u}}^{k} \|^2  \\
&+ C ( \| \Delta \tilde{\textbf{u}}^{k} \|^2+ \| \nabla \tilde{\textbf{u}}^{k} \|^2 )C_0^4 (\Delta t)^5.
\endaligned
\end{equation}
Using Cauchy-Schwarz inequality, the last term on the right hand side of \eqref{e_error_step2_5} can be bounded by
\begin{equation}\label{e_error_step2_11}
\aligned
&  2 \Delta t \| \Delta \tilde{e}_{\textbf{u}}^{k+1}  \| ( \|p(t^k)-p(t^{k+1}) \|+\| \textbf{S}_{\textbf{u}}^{k+1} \|+ \| (\textbf{u}(t^{k+1})\cdot \nabla)\textbf{u}(t^{k+1}) - ( \textbf{u}(t^{k})\cdot \nabla)\textbf{u}(t^{k})  \|  ) \\
\leq & \frac{ (1-\alpha) \nu}{6}  \Delta t \| \Delta \tilde{e}_{\textbf{u}}^{k+1} \|^2 + C (\Delta t)^2 \left( \int_{t^n}^{t^{k+1}} \| p_t \|^2 dt + \int_{t^n}^{t^{k+1}}\|\textbf{u}_{tt}\|^2 dt 
\right) \\
& + C (\Delta t)^2 \left(  \int_{t^n}^{t^{k+1}}\| \nabla \textbf{u}_{t}\|^2 dt \| \textbf{u}(t^{k+1}) \|_{H^2}^2 +\| \textbf{u}(t^{k}) \|_{H^1}^2 \int_{t^n}^{t^{k+1}}\| \textbf{u}_{t}\|_{H^2}^2 dt 
\right) .
\endaligned
\end{equation}
Finally,  combining \eqref{e_error_step2_5} with \eqref{e_error_step2_6}-\eqref{e_error_step2_11}, we obtain
\begin{equation}\label{e_error_step2_12}
\aligned
& ( \| \nabla \tilde{e}_{\textbf{u}}^{k+1} \|^2 - \| \nabla \tilde{e}_{\textbf{u}}^{k} \|^2 + \| \nabla \tilde{e}_{\textbf{u}}^{k+1}- \nabla \tilde{e}_{\textbf{u}}^{k} \|^2) + (1-  \frac{ 1-\alpha}{3})  \nu \Delta t \| \Delta \tilde{e}_{\textbf{u}}^{k+1} \|^2 \\
\leq &  ( \alpha +  \frac{ 1-\alpha }{6}  ) \nu  \Delta t \| \Delta \tilde{e}_\textbf{u} ^{k} \|^2+ \nu  C_{\alpha} \Delta t \| \nabla  \tilde{e}_\textbf{u} ^{k} \|^2+ C \| \nabla \tilde{e}_{\textbf{u}}^{k} \|^2  \\
&  + C ( \| \Delta \tilde{\textbf{u}}^{k} \|^2+ \| \nabla \tilde{\textbf{u}}^{k} \|^2 )C_0^4 (\Delta t)^5  + C (\Delta t)^2 \left( \int_{t^n}^{t^{k+1}} \| p_t \|^2 dt + \int_{t^n}^{t^{k+1}}\|\textbf{u}_{tt}\|^2 dt \right) \\
& + C (\Delta t)^2 \left(  \int_{t^n}^{t^{k+1}}\| \nabla \textbf{u}_{t}\|^2 dt \| \textbf{u}(t^{k+1}) \|_{H^2}^2 +\| \textbf{u}(t^{k}) \|_{H^1}^2 \int_{t^n}^{t^{k+1}}\| \textbf{u}_{t}\|_{H^2}^2 dt 
\right).
\endaligned
\end{equation}
Summing \eqref{e_error_step2_12} over $k$, $k=0,2,\ldots,n$, using \eqref{e_error_velo11}, \eqref{e_error_velo14}
and lemma \ref{lem: gronwall2}, we can arrive at
\begin{equation}\label{e_error_step2_13}
\aligned
& \|  \nabla \tilde{e}_{\textbf{u}}^{n+1} \|^2  + \Delta t \sum\limits_{k=0}^{n} \| \Delta \tilde{e}_{\textbf{u}}^{k+1} \|^2  \leq 
\begin{cases}
C_1 \left(1+C_0^4 (\Delta t)^2 \right) (\Delta t)^2, \ \ d=2,  \ \forall n\leq T/\Delta t, \\
C_1 \left(1+C_0^4 (\Delta t)^2 \right) (\Delta t)^2,  \ \ d=3, \ \forall n\leq T_*/\Delta t,
\end{cases}
\endaligned
\end{equation}  
where $C_1$ is independent of $C_0$ and $\Delta t$.

Next we estimate $\|  \tilde{e}_{\textbf{u}}^{n+1} \|$. Taking the inner product of \eqref{e_error_step2_2} with $ 2 \Delta t \tilde{e}_{\textbf{u}}^{k+1}$ and using the similar procedure as above, we can easily obtain that
\begin{equation}\label{e_error_step2_14}
\aligned
 \|  \tilde{e}_{\textbf{u}}^{n+1} \|^2  + \Delta t \sum\limits_{k=0}^{n} \| \nabla \tilde{e}_{\textbf{u}}^{k+1} \|^2  \leq  & C_2 \Delta t \sum\limits_{k=0}^{n}  \|  \tilde{e}_{\textbf{u}}^{n+1} \|^2 + C \Delta t \sum\limits_{k=0}^{n} \| \Delta \tilde{e}_{\textbf{u}}^{k+1} \|^2 \\
&  + C ( \| \Delta \tilde{\textbf{u}}^{k} \|^2+ \| \nabla \tilde{\textbf{u}}^{k} \|^2 )C_0^4 (\Delta t)^4  + C (\Delta t)^2.
\endaligned
\end{equation}  
where $C_2$ is independent of $C_0$ and $\Delta t$.
Choosing $\Delta t \leq \frac{1}{2C_2}$, using \eqref{e_error_step2_13} and discrete Gronwall inequality, we have
\begin{equation}\label{e_error_step2_15}
\aligned
& \|  \tilde{e}_{\textbf{u}}^{n+1} \|^2  + \Delta t \sum\limits_{k=0}^{n} \| \nabla \tilde{e}_{\textbf{u}}^{k+1} \|^2  \leq  
\begin{cases}
C_3 \left(1+C_0^4 (\Delta t)^2 \right) (\Delta t)^2, \ \ d=2,  \ \forall n\leq T/\Delta t, \\
C_3 \left(1+C_0^4 (\Delta t)^2 \right) (\Delta t)^2,  \ \ d=3, \ \forall n\leq T_*/\Delta t,
\end{cases}
\endaligned
\end{equation}  
where $C_3$ is independent of $C_0$ and $\Delta t$.

\noindent{\bf Step 3: Estimates for $| 1-\xi^{n+1} |$.} 
We shall first start by establishing  an error equation corresponding to \eqref{e_model_semi2}. 
 Let $\textbf{S}_{ R }^{k+1}$ be the truncation error defined by
\begin{equation}\label{e_error_step3_1}
\aligned
\textbf{S}_{ R }^{k+1}=\frac{\partial R (t^{k+1})}{\partial t}- \frac{ R (t^{k+1})-R(t^{k})}{\Delta t}=\frac{1}{\Delta t}\int_{t^k }^{t^{k+1}}(t^k-t)\frac{\partial^2 R }{\partial t^2}dt.
\endaligned
\end{equation}
Subtracting \eqref{e_model_transform2} at $t^{k+1}$ from \eqref{e_model_semi2}, we obtain
\begin{equation}\label{e_error_step3_2}
\aligned
\frac{ e_{R}^{k+1} -  e_{R}^{k} }{2\Delta t} = & \frac{R^{k+1} }{ E( \tilde{\textbf{u}}^{k+1} ) +K_0} \left(  -\nu \| \nabla \tilde{\textbf{u}}^{k+1} \|^2 + ( \textbf{f}^{k+1}, \tilde{\textbf{u}}^{k+1} ) \right) \\
& - \frac{R(t^{k+1} )}{ E( \textbf{u}(t^{k+1}) ) +K_0} \left(  -\nu \| \nabla \textbf{u}(t^{k+1}) \|^2 + ( \textbf{f}^{k+1}, \textbf{u}(t^{k+1}) ) \right) + \textbf{S}_{ R }^{k+1} .
\endaligned
\end{equation}
Using \eqref{e_stability1}, the first two terms on the right hand side of \eqref{e_error_step3_2} can be estimated by
\begin{equation}\label{e_error_step3_3}
\aligned
& \frac{R^{k+1} }{ E( \tilde{\textbf{u}}^{k+1} ) +K_0} \left(  -\nu \| \nabla \tilde{\textbf{u}}^{k+1} \|^2 + ( \textbf{f}^{k+1}, \tilde{\textbf{u}}^{k+1} ) \right) \\
& - \frac{R(t^{k+1} )}{ E( \textbf{u}(t^{k+1}) ) +K_0} \left(  -\nu \| \nabla \textbf{u}(t^{k+1}) \|^2 + ( \textbf{f}^{k+1}, \textbf{u}(t^{k+1}) ) \right) \\
= &  \frac{R^{k+1} }{ E( \tilde{\textbf{u}}^{k+1} ) +K_0} \left( \nu \| \nabla \textbf{u}(t^{k+1}) \|^2 - \nu \| \nabla \tilde{\textbf{u}}^{k+1} \|^2  + ( \textbf{f}^{k+1}, \tilde{e}_{\textbf{u}}^{k+1} )\right) \\
&+ \left( \frac{R^{k+1} }{ E( \tilde{\textbf{u}}^{k+1} ) +K_0} - \frac{R(t^{k+1} )}{ E( \textbf{u}(t^{k+1}) ) +K_0}  \right)
\left(  -\nu \| \nabla \textbf{u}(t^{k+1}) \|^2 + ( \textbf{f}^{k+1}, \textbf{u}(t^{k+1}) ) \right) \\
\leq & C \| \nabla \tilde{ \textbf{u} }^{k+1} \| \| \nabla \tilde{e}_{\textbf{u}}^{k+1} \| + C  \|  \tilde{e}_{\textbf{u}}^{k+1} \| +
C | E( \textbf{u}(t^{k+1}) ) - E( \tilde{\textbf{u}}^{k+1} ) | + C |e_{R}^{k+1} | \\
\leq & C \| \nabla \tilde{ \textbf{u} }^{k+1} \| \| \nabla \tilde{e}_{\textbf{u}}^{k+1} \| + C  \|  \tilde{e}_{\textbf{u}}^{k+1} \| +
C \| \tilde{ \textbf{u} }^{k+1} \| \|  \tilde{e}_{\textbf{u}}^{k+1} \|+ C |e_{R}^{k+1} |.
\endaligned
\end{equation}
Then taking the inner product of  \eqref{e_error_step3_2} with $2 \Delta t e_{R}^{k+1}$ leads to
\begin{equation}\label{e_error_step3_4}
\aligned
& ( | e_{R}^{k+1} |^2 - | e_{R}^{k} |^2 + | e_{R}^{k+1}- e_{R}^{k} |^2)  \\
\leq &  C \Delta t | e_{R}^{k+1} |^2 + C \Delta t  \| \nabla \tilde{e}_{\textbf{u}}^{k+1} \|^2 + C  \Delta t \|  \tilde{e}_{\textbf{u}}^{k+1} \|^2 
\endaligned
\end{equation}
Summing \eqref{e_error_step3_4} over $k$, $k=0,2,\ldots,n$ and using \eqref{e_error_step2_13} and \eqref{e_error_step2_15} lead to
\begin{equation}\label{e_error_step3_5}
\aligned
 | e_{R}^{n+1} |^2 \leq & C_4 \Delta t \sum\limits_{k=0}^{n} | e_{R}^{k+1} |^2 + 
C \Delta t \sum\limits_{k=0}^{n} \| \nabla \tilde{e}_{\textbf{u}}^{k+1} \|^2 + 
C \Delta t \sum\limits_{k=0}^{n} \| \tilde{e}_{\textbf{u}}^{k+1} \|^2 \\
\leq &
\begin{cases}
 C_4 \Delta t \sum\limits_{k=0}^{n} | e_{R}^{k+1} |^2 + C \left(1+C_0^4 (\Delta t)^2 \right) (\Delta t)^2, \ \ d=2,  \ \forall n\leq T/\Delta t, \\
 C_4 \Delta t \sum\limits_{k=0}^{n} | e_{R}^{k+1} |^2  + C \left(1+C_0^4 (\Delta t)^2 \right) (\Delta t)^2,  \ \ d=3, \ \forall n\leq T_*/\Delta t,
 \end{cases}
\endaligned
\end{equation}
where $C_4$ and $C$ are independent of $C_0$ and $\Delta t$. Thus choosing $\Delta t \leq \frac{1}{2C_4}$ and using discrete Gronwall inequality, we have
\begin{equation}\label{e_error_step3_6}
\aligned
 | e_{R}^{n+1} |^2 \leq & 
\begin{cases}
 C_5 \left(1+C_0^4 (\Delta t)^2 \right) (\Delta t)^2, \ \ d=2,  \ \forall n\leq T/\Delta t, \\
 C_5 \left(1+C_0^4 (\Delta t)^2 \right) (\Delta t)^2,  \ \ d=3, \ \forall n\leq T_*/\Delta t,
 \end{cases}
\endaligned
\end{equation}
where $C_5$ is independent of $C_0$ and $\Delta t$.

Next we  finish the induction process as follows. Recalling \eqref{e_model_semi3}, we have
\begin{equation}\label{e_error_step3_7}
\aligned
 | 1- \xi^{n+1} | =  & | \frac{R(t^{k+1} )}{ E( \textbf{u}(t^{k+1}) ) +K_0}  - \frac{R^{n+1} }{ E( \tilde{\textbf{u}}^{n+1} ) +K_0} | \\
 \leq & C(  | e_{R}^{n+1} | + \| \tilde{e}_{\textbf{u}}^{n+1} \| ) \\
  \leq &
\begin{cases}
 C_6 \Delta t \sqrt{  1+C_0^4 (\Delta t)^2 }, \ \ d=2,  \ \forall n\leq T/\Delta t, \\
 C_6 \Delta t \sqrt{1+C_0^4 (\Delta t)^2 },  \ \ d=3, \ \forall n\leq T_*/\Delta t,
 \end{cases}
\endaligned
\end{equation}
where $C_6$ is independent of $C_0$ and $\Delta t$.

Let $C_0=\max\{ 2 C_6, \sqrt{2C_4 }, \sqrt{2C_2 }, 4\}$ and $\Delta t \leq \frac{1}{ 1+C_0^2 }$,  we can obtain
\begin{equation}\label{e_error_step3_8}
\aligned
 C_6 \sqrt{  1+C_0^4 (\Delta t)^2 } \leq C_6 (1+C_0^2 \Delta t ) \leq C_0.
\endaligned
\end{equation}
Then combining \eqref{e_error_step3_7} with \eqref{e_error_step3_8} results in
\begin{equation}\label{e_error_step3_9}
\aligned
 | 1- \xi^{n+1} | 
  \leq &
\begin{cases}
 C_0 \Delta t, \ \ d=2, \ \forall n\leq T/\Delta t, \\
 C_0 \Delta t,  \ \ d=3, \ \forall n\leq T_*/\Delta t,
 \end{cases}
\endaligned
\end{equation}
which completes the induction process \eqref{e_error_velo1}.

Now combining \eqref{e_error_step2_13} with \eqref{e_error_step2_15}, we have
\begin{equation}\label{e_error_step3_10}
\aligned
& \|  \tilde{e}_{\textbf{u}}^{n+1} \|^2 + \|  \nabla \tilde{e}_{\textbf{u}}^{n+1} \|^2  + \Delta t \sum\limits_{k=0}^{n} \| \Delta \tilde{e}_{\textbf{u}}^{k+1} \|^2  \leq 
\begin{cases}
C (\Delta t)^2, \ \ d=2, \ \forall n\leq T/\Delta t, \\
C (\Delta t)^2,  \ \ d=3, \ \forall n\leq T_*/\Delta t.
\end{cases}
\endaligned
\end{equation} 

Noting \eqref{e_model_semi3} and \eqref{e_error_step3_9}, and using the stability results \eqref{e_error_velo14} and \eqref{e_error_velo15}, we have
\begin{equation}\label{e_error_step3_11}
\aligned
& \|  e_{\textbf{u}}^{n+1} \|^2 + \|  \nabla e_{\textbf{u}}^{n+1} \|^2  + \Delta t \sum\limits_{k=0}^{n} \| \Delta e_{\textbf{u}}^{k+1} \|^2  \\
\leq & 2 ( \|  \tilde{e}_{\textbf{u}}^{n+1} \|^2+ |\xi^{n+1} -1|^4 \| \tilde{\textbf{u}}^{n+1} \|^2 )
+ 2 ( \|  \nabla\tilde{e}_{\textbf{u}}^{n+1} \|^2+|\xi^{n+1} -1|^4 \| \nabla \tilde{\textbf{u}}^{n+1} \|^2 ) \\
&+2 \Delta t \sum\limits_{k=0}^{n} ( \| \Delta \tilde{e}_{\textbf{u}}^{k+1} \|^2 +  |\xi^{n+1} -1|^4 \| \Delta \tilde{\textbf{u}}^{n+1} \|^2 ) \\
\leq & 2 ( \|  \tilde{e}_{\textbf{u}}^{n+1} \|^2+ |\xi^{n+1} -1|^4 \| \tilde{\textbf{u}}^{n+1} \|^2 )
+ 2 ( \|  \nabla\tilde{e}_{\textbf{u}}^{n+1} \|^2+ |\xi^{n+1} -1|^4 \| \nabla \tilde{\textbf{u}}^{n+1} \|^2 ) \\
&+2 \Delta t \sum\limits_{k=0}^{n} ( \| \Delta \tilde{e}_{\textbf{u}}^{k+1} \|^2 +  |\xi^{n+1} -1|^4 \| \Delta \tilde{\textbf{u}}^{n+1} \|^2 ) \\
\leq &
\begin{cases}
C (\Delta t)^2, \ \ d=2,  \ \forall n\leq T/\Delta t, \\
C (\Delta t)^2,  \ \ d=3, \ \forall n\leq T_*/\Delta t,
\end{cases}
\endaligned
\end{equation} 
which leads to the desired results \eqref{e_error_estimate_u_final}.

It remains to  estimate the  pressure error. Recalling \eqref{e_error_step2_4}, we can transform \eqref{e_error_step2_pre} into the following:
\begin{equation}\label{e_error_step3_12}
\aligned
 \Delta t \sum\limits_{k=0}^{n} \| \nabla e_p^{k+1} \|^2 \leq &  C \Delta t \sum\limits_{k=0}^{n}  \| \Delta \tilde{e}_\textbf{u} ^{k+1} \|^2 + C \Delta t \sum\limits_{k=0}^{n}  \| \nabla  \tilde{e}_\textbf{u} ^{k+1} \|^2 \\
 & +  C \Delta t \sum\limits_{k=0}^{n}  \| \nabla e_{ \textbf{u} }^{k+1} \|^2 \| \nabla \textbf{u}(t^{k+1} ) \|  \| \nabla \textbf{u}(t^{k+1})  \|_{H^1} \\
 & + C \Delta t \sum\limits_{k=0}^{n} \Delta t \| \nabla \textbf{u}^{k+1} \|^2 \| \nabla e_{ \textbf{u} }^{k+1}  \|  \| \nabla e_{ \textbf{u} }^{k+1}  \|_{H^1} \\
\leq &
\begin{cases}
C (\Delta t)^2, \ \ d=2, \ \forall n\leq T/\Delta t, \\
C (\Delta t)^2,  \ \ d=3, \ \forall n\leq T_*/\Delta t.
\end{cases}
\endaligned
\end{equation} 

Taking $ q = \Delta^{-1} e_p^{k+1} $ in \eqref{e_error_step2_3} and using \eqref{e_estimate for trilinear form}, we can obtain 
\begin{equation}\label{e_error_step3_13}
\aligned
 \|  e_p^{k+1} \|^2 = & \left(  ( \textbf{u}(t^{k+1} ) \cdot \nabla ) \textbf{u}( t^{k+1} ) - ( \textbf{u}^{k+1} \cdot \nabla ) \textbf{u}^{k+1} , \Delta^{-1} e_p^{k+1} ) \right) \\
& - ( \nu \nabla \times \nabla \times \tilde{e}_\textbf{u} ^{k+1}, \Delta^{-1} e_p^{k+1} ) \\
\leq & C \|  e_{ \textbf{u} }^{k+1} \|_1 \| \nabla \textbf{u}(t^{k+1} )  \|_1 + C \| \textbf{u}^{k+1} \|_1 \|  e_{ \textbf{u} }^{k+1} \|_1 \\
&+C \| \nabla e_{ \textbf{u} }^{k+1} \|^2 
+ \frac{1}{2}  \|  e_p^{k+1} \|^2 \\
\leq &
\begin{cases}
C (\Delta t)^2, \ \ d=2, \ \forall n\leq T/\Delta t, \\
C (\Delta t)^2,  \ \ d=3, \ \forall n\leq T_*/\Delta t,
\end{cases}
\endaligned
\end{equation} 
which leads to the desired results \eqref{e_error_estimate_p_final}.
\end{proof}

 \section{Numerical experiments and concluding remarks}
 We present in this section some numerical experiments followed by some concluding remarks.
  \subsection{Numerical results}
We first present some numerical tests  to verify the accuracy of the first-order GSAV scheme with consistent splitting method \eqref{e_model_semi1}-\eqref{e_model_semi4} for the Navier-Stokes equations.
In all examples below, we take $\Omega=(0,1)\times(0,1)$.
We set $T=1$, $K_0=1$ and the spatial discretization is based on the MAC scheme on the staggered grid with $N_x=N_y=250$ so that the spatial discretization error is negligible compared to the time discretization error for the time steps used in the experiments.

{\bf Example 1}. The right hand side of the equations are computed according to the analytic solution given by:\\
\begin{equation*}
\aligned
\begin{cases}
p(x,y,t)=t(x^3-0.25),\\
u_1(x,y,t)=-t x^2(x-1)^2y(y-1)(2y-1),\\
u_2(x,y,t)=t y^2(y-1)^2x(x-1)(2x-1).
\end{cases}
\endaligned
\end{equation*}

{\bf Example 2}. The right hand side of the equations are computed according to the analytic solution given by:\\
\begin{equation*}
\aligned
\begin{cases}
p(x,y,t)=\sin(t)(\sin(\pi y)-2/\pi),\\
u_1(x,y,t)= \sin(t)\sin^2(\pi x)\sin(2\pi y),\\
u_2(x,y,t)=- \sin(t)\sin(2\pi x)\sin^2(\pi y).
\end{cases}
\endaligned
\end{equation*}
We demonstrate numerical results for Examples 1 and 2 with different viscosity coefficients $\nu=1,0.1,0.01$ in Tables \ref{table1_example1}-\ref{table3_example2}. It can be easily observed that the numerical results for the velocity and pressure in different norms are all consistent with the error estimates in Theorem  \ref{the: error_estimate_final}.

\begin{table}[htbp]
\renewcommand{\arraystretch}{1.1}
\small
\centering
\caption{Errors and convergence rates for Example 1 with $\nu=1$. }\label{table1_example1}
\begin{tabular}{p{1.0cm}p{1.2cm}p{1.2cm}p{1.2cm}p{1.2cm}p{1.2cm}p{1.2cm}p{1.2cm}p{1.2cm}p{1.2cm}p{1.2cm} }\hline
$\Delta t$    &$\|e_{\textbf{u}}\|_{l^{\infty}}$    &Rate &$\|\nabla e_{\textbf{u}} \|_{l^{\infty} }$   &Rate &$\|e_{p}\|_{l^{\infty}}$    &Rate &$\|\nabla e_{p} \|_{l^{2} }$   &Rate
  \\ \hline
$1/10 $    &8.45E-3  & ---    &4.30E-2   &---  &5.27E-2     &--- &2.85E-1   &---    \\
$1/20$    &4.43E-3   &0.93    &2.28E-2   &0.92 &2.87E-2   &0.88
&1.92E-1     &0.57  \\
$1/40$    &2.24E-3  &0.98     &1.15E-2   &0.98 &1.46E-2   &0.97
&1.12E-1         &0.78   \\
$1/80$    &1.12E-3  &1.00    &5.78E-3    &1.00   &7.32E-3    &1.00
&5.98E-2         &0.90   \\
\hline
\end{tabular}
\end{table}

\begin{table}[htbp]
\renewcommand{\arraystretch}{1.1}
\small
\centering
\caption{Errors and convergence rates for Example 1 with $\nu=0.1$. }\label{table2_example1}
\begin{tabular}{p{1.0cm}p{1.2cm}p{1.2cm}p{1.2cm}p{1.2cm}p{1.2cm}p{1.2cm}p{1.2cm}p{1.2cm}p{1.2cm}p{1.2cm} }\hline
$\Delta t$    &$\|e_{\textbf{u}}\|_{l^{\infty}}$    &Rate &$\|\nabla e_{\textbf{u}} \|_{l^{\infty} }$   &Rate &$\|e_{p}\|_{l^{\infty}}$    &Rate &$\|\nabla e_{p} \|_{l^{2} }$   &Rate
  \\ \hline
$1/10 $    &5.10E-2  & ---    &2.74E-1   &---  &3.36E-2     &--- &1.86E-1       &--- \\
$1/20$    &2.71E-2   &0.91    &1.47E-1   &0.90 &1.85E-2   &0.86
&1.19E-1     &0.64  \\
$1/40$    &1.40E-2  &0.95     &7.62E-2   &0.95 &9.70E-3   &0.93
&6.88E-2         &0.80   \\
$1/80$    &7.10E-3  &0.98    &3.88E-2    &0.97   &4.95E-3    &0.97
&3.69E-2       &0.90   \\
\hline
\end{tabular}
\end{table}

\begin{table}[htbp]
\renewcommand{\arraystretch}{1.1}
\small
\centering
\caption{Errors and convergence rates for Example 1 with $\nu=0.01$. }\label{table3_example1}
\begin{tabular}{p{1.0cm}p{1.2cm}p{1.2cm}p{1.2cm}p{1.2cm}p{1.2cm}p{1.2cm}p{1.2cm}p{1.2cm}p{1.2cm}p{1.2cm} }\hline
$\Delta t$    &$\|e_{\textbf{u}}\|_{l^{\infty}}$    &Rate &$\|\nabla e_{\textbf{u}} \|_{l^{\infty} }$   &Rate &$\|e_{p}\|_{l^{\infty}}$    &Rate &$\|\nabla e_{p} \|_{l^{2} }$   &Rate
  \\ \hline
$1/10 $    &1.00E-1  & ---    &7.66E-1   &---  &1.45E-2     &--- &8.36E-2       &--- \\
$1/20$    &5.11E-2   &0.97    &3.97E-1   &0.95 &6.41E-3   &1.18
&4.76E-2     &0.81  \\
$1/40$    &2.58E-2  &0.98     &2.03E-1   &0.97 &3.01E-3   &1.09
&2.63E-2         &0.86   \\
$1/80$    &1.30E-2  &0.99    &1.02E-1    &0.99   &1.46E-3    &1.04
&1.39E-2       &0.92   \\
\hline
\end{tabular}
\end{table}

\begin{table}[htbp]
\renewcommand{\arraystretch}{1.1}
\small
\centering
\caption{Errors and convergence rates for Example 2 with $\nu=1$. }\label{table1_example2}
\begin{tabular}{p{1.0cm}p{1.2cm}p{1.2cm}p{1.2cm}p{1.2cm}p{1.2cm}p{1.2cm}p{1.2cm}p{1.2cm}p{1.2cm}p{1.2cm} }\hline
$\Delta t$    &$\|e_{\textbf{u}}\|_{l^{\infty}}$    &Rate &$\|\nabla e_{\textbf{u}} \|_{l^{\infty} }$   &Rate &$\|e_{p}\|_{l^{\infty}}$    &Rate &$\|\nabla e_{p} \|_{l^{2} }$   &Rate
  \\ \hline
$1/10 $    &6.59E-3  & ---    &5.04E-2   &---  &4.17E-2     &--- &4.06E-1   &---    \\
$1/20$    &3.24E-3   &1.02    &2.50E-2   &1.01 &2.15E-2   &0.96
&2.62E-1     &0.63  \\
$1/40$    &1.59E-3  &1.03     &1.22E-2   &1.03 &1.04E-2   &1.04
&1.48E-1         &0.82   \\
$1/80$    &7.86E-4  &1.02    &6.02E-3    &1.02   &5.05E-3    &1.04
&7.79E-2         &0.93   \\
\hline
\end{tabular}
\end{table}

\begin{table}[htbp]
\renewcommand{\arraystretch}{1.1}
\small
\centering
\caption{Errors and convergence rates for Example 2 with $\nu=0.1$. }\label{table2_example2}
\begin{tabular}{p{1.0cm}p{1.2cm}p{1.2cm}p{1.2cm}p{1.2cm}p{1.2cm}p{1.2cm}p{1.2cm}p{1.2cm}p{1.2cm}p{1.2cm} }\hline
$\Delta t$    &$\|e_{\textbf{u}}\|_{l^{\infty}}$    &Rate &$\|\nabla e_{\textbf{u}} \|_{l^{\infty} }$   &Rate &$\|e_{p}\|_{l^{\infty}}$    &Rate &$\|\nabla e_{p} \|_{l^{2} }$   &Rate
  \\ \hline
$1/10 $    &6.48E-2  & ---    &4.59E-1   &---  &4.10E-2     &--- &3.45E-1       &--- \\
$1/20$    &3.34E-2   &0.95    &2.58E-1   &0.94 &2.26E-2   &0.86
&2.05E-1     &0.66  \\
$1/40$    &1.69E-2  &0.98     &1.32E-1   &0.97 &1.17E-2   &0.95
&1.17E-1         &0.81   \\
$1/80$    &8.52E-3  &0.99    &6.62E-2    &0.99   &5.90E-3    &0.98
&6.21E-2       &0.91   \\
\hline
\end{tabular}
\end{table}

\begin{table}[htbp]
\renewcommand{\arraystretch}{1.1}
\small
\centering
\caption{Errors and convergence rates for Example 2 with $\nu=0.01$. }\label{table3_example2}
\begin{tabular}{p{1.0cm}p{1.2cm}p{1.2cm}p{1.2cm}p{1.2cm}p{1.2cm}p{1.2cm}p{1.2cm}p{1.2cm}p{1.2cm}p{1.2cm} }\hline
$\Delta t$    &$\|e_{\textbf{u}}\|_{l^{\infty}}$    &Rate &$\|\nabla e_{\textbf{u}} \|_{l^{\infty} }$   &Rate &$\|e_{p}\|_{l^{\infty}}$    &Rate &$\|\nabla e_{p} \|_{l^{2} }$   &Rate
  \\ \hline
$1/10 $    &2.54E-1  & ---    &2.50E-0   &---  &1.36E-1     &--- &5.96E-1       &--- \\
$1/20$    &1.37E-1   &0.89    &1.36E-0   &0.88 &7.57E-2   &0.85
&3.01E-1     &0.98  \\
$1/40$    &7.07E-2  &0.95     &7.00E-1   &0.96 &3.89E-2   &0.96
&1.48E-1         &1.02   \\
$1/80$    &3.59E-2  &0.98    &3.54E-1    &0.98   &1.96E-2    &0.99
&7.32E-2       &1.02   \\
\hline
\end{tabular}
\end{table}

\subsection{Concluding remarks}

We carried out a rigorous error analysis of the first-order semi-discrete (in time) consistent splitting GSAV scheme for the Navier-Stokes equations with no-slip boundary conditions. The scheme is linear,  unconditionally stable, and only requires solving a sequence of  Poisson type equations at each time step.  Thanks to its unconditional stability, we were able to derive  optimal global  (resp. local) in time  error estimates in the two (resp. three) dimensional case for the velocity and pressure approximations.  To the best of our knowledge, this is the first global in time  error estimate for a consistent splitting scheme for the  Navier-Stokes equations with no-slip boundary conditions. 

Although we only considered  semi-discrete (in time) case in this paper, the analysis can be extended, albeit tedious, to fully discrete  approximations with $C^1$ subspaces for the velocity and $C^0$ subspaces for the pressure similarly as in \cite{liu2009error}. The consistent splitting GSAV scheme can also be easily extended to higher-order \cite{wu2022new}. However, it is a non-trivial matter to extend the current  error analysis to high-order, which will be a   subject of future study.


\bibliographystyle{siamplain}
\bibliography{NSE_CS}

\begin{thebibliography}{10}

\bibitem{brezzi2012mixed}
{\sc F.~Brezzi and M.~Fortin}, {\em Mixed and hybrid finite element methods},
  vol.~15, Springer Science \& Business Media, 2012.

\bibitem{chorin1968numerical}
{\sc A.~J. Chorin}, {\em Numerical solution of the {N}avier-{S}tokes
  equations}, Mathematics of Computation, 22 (1968), pp.~745--762.

\bibitem{E2003Gauge}
{\sc W.~E and J.-G. Liu}, {\em Gauge method for viscous incompressible flows},
  Communications in Mathematical Sciences, 1 (2003), pp.~317--332.

\bibitem{elman2014finite}
{\sc H.~C. Elman, D.~J. Silvester, and A.~J. Wathen}, {\em Finite elements and
  fast iterative solvers: with applications in incompressible fluid dynamics},
  Oxford University Press, USA, 2014.

\bibitem{girault1979finite}
{\sc V.~Girault and P.-A. Raviart}, {\em Finite element approximation of the
  {N}avier-{S}tokes equations}, Lecture Notes in Mathematics, Berlin Springer
  Verlag, 749 (1979).

\bibitem{glowinski2003finite}
{\sc R.~Glowinski}, {\em Finite element methods for incompressible viscous
  flow}, Handbook of Numerical Analysis, 9 (2003), pp.~3--1176.

\bibitem{guermond2003new}
{\sc J.~Guermond and J.~Shen}, {\em A new class of truly consistent splitting
  schemes for incompressible flows}, Journal of Computational Physics, 192
  (2003), pp.~262--276.

\bibitem{guermond2004error}
{\sc J.~Guermond and J.~Shen}, {\em On the error estimates for the rotational
  pressure-correction projection methods}, Mathematics of Computation, 73
  (2004), pp.~1719--1737.

\bibitem{GMS06}
{\sc J.~L. Guermond, P.~Minev, and J.~Shen}, {\em An overview of projection
  methods for incompressible flows}, Comput. Methods Appl. Mech. Engrg., 195
  (2006), pp.~6011--6045.

\bibitem{guermond2003velocity}
{\sc J.-L. Guermond and J.~Shen}, {\em Velocity-correction projection methods
  for incompressible flows}, SIAM Journal on Numerical Analysis, 41 (2003),
  pp.~112--134.

\bibitem{gunzburger2012finite}
{\sc M.~D. Gunzburger}, {\em Finite element methods for viscous incompressible
  flows: a guide to theory, practice, and algorithms}, Elsevier, 2012.

\bibitem{HeSu07}
{\sc Y.~He and W.~Sun}, {\em Stability and convergence of the
  {C}rank-{N}icolson/{A}dams-{B}ashforth scheme for the time-dependent
  {N}avier-{S}tokes equations}, SIAM J. Numer. Anal., 45 (2007), pp.~837--869.

\bibitem{heywood1982finite}
{\sc J.~G. Heywood and R.~Rannacher}, {\em Finite element approximation of the
  nonstationary {N}avier-{S}tokes problem. {I}. regularity of solutions and
  second-order error estimates for spatial discretization}, SIAM Journal on
  Numerical Analysis, 19 (1982), pp.~275--311.

\bibitem{huang2021stability}
{\sc F.~Huang and J.~Shen}, {\em Stability and error analysis of a class of
  high-order {IMEX} schemes for {N}avier--{S}tokes equations with periodic
  boundary conditions}, SIAM Journal on Numerical Analysis, 59 (2021),
  pp.~2926--2954.

\bibitem{MR4383075}
{\sc F.~Huang and J.~Shen}, {\em A new class of implicit-explicit {BDF{$k$}}
  {SAV} schemes for general dissipative systems and their error analysis},
  Comput. Methods Appl. Mech. Engrg., 392 (2022), pp.~Paper No. 114718, 25.

\bibitem{Johnston2004Accurate}
{\sc H.~Johnston and J.-G. Liu}, {\em Accurate, stable and efficient
  {N}avier–{S}tokes solvers based on explicit treatment of the pressure
  term}, Journal of Computational Physics, 199 (2004), pp.~221--259.

\bibitem{ladyzhenskaya1969mathematical}
{\sc O.~A. Ladyzhenskaya}, {\em The mathematical theory of viscous
  incompressible flow}, Gordon \& Breach,  (1969).

\bibitem{li2020error}
{\sc X.~Li and J.~Shen}, {\em Error analysis of the {SAV}-{MAC} scheme for the
  {N}avier-{S}tokes equations}, SIAM Journal on Numerical Analysis, 58 (2020),
  pp.~2465--2491.

\bibitem{li2022new}
{\sc X.~Li, J.~Shen, and Z.~Liu}, {\em New {SAV}-pressure correction methods
  for the {N}avier-{S}tokes equations: stability and error analysis},
  Mathematics of Computation, 91 (2022), pp.~141--167.

\bibitem{lin2019numerical}
{\sc L.~Lin, Z.~Yang, and S.~Dong}, {\em Numerical approximation of
  incompressible {N}avier-{S}tokes equations based on an auxiliary energy
  variable}, Journal of Computational Physics, 388 (2019), pp.~1--22.

\bibitem{liu2007stability}
{\sc J.-G. Liu, J.~Liu, and R.~L. Pego}, {\em Stability and convergence of
  efficient {N}avier-{S}tokes solvers via a commutator estimate},
  Communications on Pure and Applied Mathematics: A Journal Issued by the
  Courant Institute of Mathematical Sciences, 60 (2007), pp.~1443--1487.

\bibitem{liu2009error}
{\sc J.-G. Liu, J.~Liu, and R.~L. Pego}, {\em Error estimates for
  finite-element {N}avier-{S}tokes solvers without standard inf-sup
  conditions}, Chinese Annals of Mathematics, Series B, 30 (2009),
  pp.~743--768.

\bibitem{liu2010stable}
{\sc J.-G. Liu and R.~Pego}, {\em Stable discretization of magnetohydrodynamics
  in bounded domains}, Communications in Mathematical Sciences, 8 (2010),
  pp.~235--251.

\bibitem{Nochetto2003Gauge}
{\sc R.~Nochetto and J.-H. Pyo}, {\em Error estimates for semi-discrete gauge
  methods for the {N}avier-{S}tokes equations}, Mathematics of computation, 74
  (2005), pp.~521--542.

\bibitem{serson2016velocity}
{\sc D.~Serson, J.~Meneghini, and S.~J. Sherwin}, {\em Velocity-correction
  schemes for the incompressible {N}avier--{S}tokes equations in general
  coordinate systems}, Journal of Computational Physics, 316 (2016),
  pp.~243--254.

\bibitem{shen1990long}
{\sc J.~Shen}, {\em Long time stability and convergence for fully discrete
  nonlinear {G}alerkin methods}, Applicable Analysis, 38 (1990), pp.~201--229.

\bibitem{shen1992error}
{\sc J.~Shen}, {\em On error estimates of projection methods for
  {N}avier-{S}tokes equations: first-order schemes}, SIAM Journal on Numerical
  Analysis, 29 (1992), pp.~57--77.

\bibitem{Shen2012Modeling}
{\sc J.~Shen}, {\em Modeling and numerical approximation of two-phase
  incompressible flows by a phase-field approach}, Multiscale Modeling and
  Analysis for Materials Simulation,  (2012), pp.~147--195.

\bibitem{shen2019new}
{\sc J.~Shen, J.~Xu, and J.~Yang}, {\em A new class of efficient and robust
  energy stable schemes for gradient flows}, SIAM Review, 61 (2019),
  pp.~474--506.

\bibitem{shen2007error}
{\sc J.~Shen and X.~Yang}, {\em Error estimates for finite element
  approximations of consistent splitting schemes for incompressible flows},
  Discrete \& Continuous Dynamical Systems-B, 8 (2007), pp.~663--675.

\bibitem{temam1969approximation}
{\sc R.~Temam}, {\em Sur l'approximation de la solution des {\'e}quations de
  {N}avier-{S}tokes par la m{\'e}thode des pas fractionnaires (ii)}, Archive
  for Rational Mechanics and Analysis, 33 (1969), pp.~377--385.

\bibitem{temam1983nonlinear}
{\sc R.~T{\'e}mam}, {\em Nonlinear functional analysis and {N}avier-{S}tokes
  equations}, SIAM, Philadelphia,  (1983).

\bibitem{Tema95}
{\sc R.~Temam}, {\em Navier-{S}tokes equations and nonlinear functional
  analysis}, vol.~66 of CBMS-NSF Regional Conference Series in Applied
  Mathematics, Society for Industrial and Applied Mathematics (SIAM),
  Philadelphia, PA, second~ed., 1995.

\bibitem{temam2001navier}
{\sc R.~Temam}, {\em {N}avier-{S}tokes equations: theory and numerical
  analysis}, vol.~343, American Mathematical Soc., 2001.

\bibitem{weinan1995projection}
{\sc E.~Weinan and J.-G. Liu}, {\em Projection method {I}: convergence and
  numerical boundary layers}, SIAM Journal on Numerical Analysis,  (1995),
  pp.~1017--1057.

\bibitem{wu2022new}
{\sc K.~Wu, F.~Huang, and J.~Shen}, {\em A new class of higher-order decoupled
  schemes for the incompressible {N}avier-{S}tokes equations and applications
  to rotating dynamics}, Journal of Computational Physics, 458 (2022),
  p.~111097.

\end{thebibliography}

\end{document}